\documentclass[a4paper,11pt,reqno]{amsart}

\usepackage{amssymb}
\usepackage{amsmath}
\usepackage{amsthm}
\usepackage{amsfonts}
\usepackage{bbm}
\usepackage{tikz}

\usepackage{bookmark}
\usepackage{hyperref}
\hypersetup{pdfstartview={FitH}}

\newtheorem{theorem}{Theorem}
\newtheorem{lemma}[theorem]{Lemma}
\newtheorem{corollary}[theorem]{Corollary}

\numberwithin{equation}{section}

\begin{document}

\title[Power-type cancellation for the simplex Hilbert transform]{Power-type cancellation for the simplex Hilbert transform}

\author[P. Durcik]{Polona Durcik}
\address{Polona Durcik, Mathematisches Institut, Universit\"at Bonn, Endenicher Allee 60, 53115 Bonn, Germany}
\email{durcik@math.uni-bonn.de}

\author[V. Kova\v{c}]{Vjekoslav Kova\v{c}}
\address{Vjekoslav Kova\v{c}, Department of Mathematics, Faculty of Science, University of Zagreb, Bijeni\v{c}ka cesta 30, 10000 Zagreb, Croatia}
\email{vjekovac@math.hr}

\author[C. Thiele]{Christoph Thiele}
\address{Christoph Thiele, Mathematisches Institut, Universit\"at Bonn, Endenicher Allee 60, 53115 Bonn, Germany}
\email{thiele@math.uni-bonn.de}

\date{\today}

\begin{abstract}
We prove $\textup{L}^p$ bounds for the truncated simplex Hilbert transform which grow with a power less than one of the truncation range in the logarithmic scale.
\end{abstract}

\maketitle


\section{Introduction}
The simplex Hilbert transform of degree $n\geq 1$ is given by
\begin{align*}
\Lambda_n:= \textup{p.v.} \int_{\mathbb{R}^{n+1}}\prod_{i=0}^{n} F_i(x_0,\dots,x_{i-1},x_{i+1},\dots,x_n)\frac{1}{x_0+\dots + x_n }dx_0 \dots dx_n .
\end{align*}
It is a multilinear form in the $n+1$ functions $F_0,\dots ,F_n$, which for simplicity we assume to be in the Schwartz class.
If $n=1$, then the simplex Hilbert transform is the form obtained by dualization of the classical Hilbert transform.
The case $n=2$ was called the triangular Hilbert transform in \cite{ktz:tht}.
A major open problem is whether for $n\geq 2$ the simplex Hilbert transform satisfies any $\textup{L}^p$ bounds of the type
\[ |\Lambda_n| \leq C \prod_{i=0}^n
 \|F_i\|_{p_i} .\]
Partial progress in the case $n=2$ was made in \cite{ktz:tht} for a dyadic model and under the
additional assumption that one of the functions $F_i$ takes certain special forms.

The papers \cite{tt:mht} and \cite{pz:splx} initiated the study of growth of the bounds for the truncated simplex Hilbert transform
\begin{align*}
\Lambda_{n,r,R} := \int_{r\leq |x_0+\dots + x_n| \leq R }\prod_{i=0}^{n} F_i(x_0,\dots,x_{i-1},x_{i+1},\dots,x_n)\frac{1}{x_0+\dots + x_n}dx_0 \dots dx_n
\end{align*}
for some truncation parameters $0<r<R$.
The trivial estimate
\begin{align}\label{est:trv}
|\Lambda_{n,r,R} | \leq 2 \Big( \log \frac{R}{r} \Big)\prod_{i=0}^n \|F_i\|_{p_i}
\end{align}
with Banach space exponents $1\leq p_i \leq \infty$ satisfying the H\"older scaling $\sum_{i=0}^n 1/p_i=1$
follows by substituting $x_0=x-x_1-\dots-x_n$, applying H\"older's inequality in $x_1,\dots,x_n$, and integrating in $x$.
Alternatively, if one is careless about the actual constant $2$, one can simply break the kernel into about $\log(R/r)$ many scales and estimate each scale separately.

Using techniques from additive combinatorics, Zorin-Kranich \cite{pz:splx} improved this bound to $o(\log (R/r))$ when $R/r\rightarrow \infty$ in the open range $1<p_i<\infty$ with the H\"older scaling. A special case of this result was shown before by Tao \cite{tt:mht}.

The main result of this paper is the following bound.
\begin{theorem} \label{thm:mainthm}
There exists a finite constant $C$ depending only on $n$ such that for any Schwartz functions $F_0,\dots, F_n$ on $\mathbb{R}^n$ and any $0<r<R$ we have \begin{align}\label{est:mainthm} |\Lambda_{n,r,R} | \leq C\Big(\log \frac{R}{r} \Big)^{1-2^{-n+1}} \|F_0\|_{{2^n}} \prod_{i=1}^n\|F_i\|_{{2^{n-i+1}}} .
\end{align}
\end{theorem}
By interpolation of \eqref{est:mainthm} with \eqref{est:trv} we obtain the following corollary.
\begin{corollary} \label{cor:maincor}
Let $1<p_0,\dots, p_n<\infty$ and $ 1/p_0 + \dots + 1/p_n =1$.
There exist a finite constant $C$ and a number $\epsilon>0$, both depending only on $n$ and $p_0,\dots, p_n$,
such that for any Schwartz functions $F_0,\dots, F_n$ on $\mathbb{R}^n$ and any $0<r<R$ we have
\[ |\Lambda_{n,r,R} | \leq C \Big( \log \frac{R}{r} \Big)^{1-\epsilon} \prod_{i=0}^n \|F_i\|_{p_i} . \]
\end{corollary}
In particular, this strengthens the results from \cite{tt:mht} and \cite{pz:splx}. The special case $n=2$ was commented on in \cite{dkst:nvea}, where it followed from boundedness of a certain square function.
A modification of our arguments could yield bounds for a simplex transform associated with more general Calder\'on-Zygmund kernels on $\mathbb{R}$ replacing $K(t)=1/t$, but we do not aim for that kind of generality here.
The reader can also consult \cite{ktz:tht} and \cite{pz:splx} for the ways of encoding various lower-dimensional or less singular operators into $\Lambda_{n}$, so that Corollary~\ref{cor:maincor} gives nontrivial estimates for the truncations of these operators too, even though some of them are already known to be (uniformly) bounded.

The proof of Theorem~\ref{thm:mainthm} is a special case of a more general estimate in Lemma~\ref{lemma:mainlemma} on auxiliary forms involving an additional parameter $1\leq k\leq n$, which is in turn proved by induction on that parameter.
The induction uses higher-dimensional analogues of the arguments in \cite{pd:L4}, \cite{dkst:nvea}, and \cite{vk:tp}, i.e.\@ intertwined applications of the Cauchy-Schwarz inequality \eqref{afterCS} and an integration by parts identity \eqref{ftc}.
The base case is closely related to the quadrilinear forms studied in \cite{pd:L4} and \cite{vk:tp}.

In Section \ref{dyadicsection} we discuss a dyadic version of Theorem \ref{thm:mainthm}.


\section{Proof of Theorem~\ref{thm:mainthm}}
We fix an integer $n\geq 2$ and numbers $0<r<R$.
One can suppose that $\log(R/r)>1$, since otherwise \eqref{est:mainthm} is even weaker than \eqref{est:trv}.
We also fix Schwartz functions $F_0,\dots, F_n$ as in Theorem~\ref{thm:mainthm}.
It is enough to work with real-valued functions, since complex-valued functions may be split into their real and imaginary parts.
By homogeneity we may assume that the functions are normalized as
\begin{align}\label{norm}
\|F_0\|_{{2^n}} = \|F_1\|_{{2^n}} = \|F_2\|_{{2^{n-1}}} = \dots = \|F_n\|_{{2^{1}}} =1.
\end{align}

Next, we pass from rough to smooth truncations of the simplex Hilbert transform.
Let us write
$$ \varphi(x):= \frac{\mathbbm{1}_{[-R,R]\setminus [-r,r]} (x) - (g(x/R)-g(x/r))}{x}
= \frac{\mathbbm{1}_{[-1,1]}(x/R) - g(x/R)}{x} - \frac{\mathbbm{1}_{[-1,1]}(x/r) - g(x/r)}{x}, $$
where $g$ is the Gaussian function $g(x):=e^{-\pi {x^2}}$.
Note that $\varphi$ is integrable uniformly in the truncation parameters $0<r<R$ and that the bound
\begin{align*}
\Big| \int_{\mathbb{R}^{n+1}}\prod_{i=0}^{n} F_i(x_0,\dots,x_{i-1},x_{i+1},\dots,x_n) \varphi(x_0+\dots + x_n) dx_0 \dots dx_n \Big| \leq \|\varphi\|_1 \|F_0\|_{{2^n}} \prod_{i=1}^n\|F_i\|_{{2^{n-i+1}}}
\end{align*}
follows from the change of variables $x_0= x-x_1-\dots - x_{n}$ and H\"older's inequality in $x_1,\dots, x_n$.
Therefore, in order to prove Theorem~\ref{thm:mainthm} it suffices to prove the estimate for the kernel
$$ \frac{g(x/R)-g(x/r)}{x} = - \int_{r}^{R} t^{-2}g'(t^{-1}x) dt.$$
That is, it suffices to obtain, in lieu of \eqref{est:mainthm},
\begin{align}
\Big| \int_r^R \int_{\mathbb{R}^{n+1}} & \prod_{i=0}^{n} F_i(x_0,\dots,x_{i-1},x_{i+1},\dots,x_n) \nonumber \\
& h_t(x_0+\dots + x_n)dx_0 \dots dx_n \frac{dt}{t} \Big|
\leq C \Big(\log \frac{R}{r} \Big)^{1-2^{-n+1}}, \label{toshow:smooth}
\end{align}
where $h$ is the derivative of $g$, and we use subscripts to denote $\textup{L}^1$-normalized dilates of functions:
\begin{align*}
h_t(x):= t^{-1}h (t^{-1}x) .
\end{align*}

For the inductive statement we need to define further expressions. For $0\leq k\leq n$ we define $\mathcal{F}^k$ as a function of variables $x_0,\dots, x_n,x_0^0,x_0^1,\dots, x_n^0,x_n^1 \in \mathbb{R}$ by
\begin{align}\label{fnk}
\mathcal{F}^k:= \prod_{i=0}^{k} \prod_{(r_{k+1},\dots,r_n)\in\{0,1\}^{n-k}}F_i(x_0,\dots,x_{i-1},x_{i+1},\dots,x_k,x_{k+1}^{r_{k+1}},\dots,x_{n}^{r_{n}}).
\end{align}
Note that $\mathcal{F}^k$ does not depend on $x_{k+1},\dots, x_n$ and $x_0^0,x_0^1,\dots, x_k^0,x_k^1$. Each factor $F_i$ in the product has the property that for each $k+1\leq j \leq n$ it is
independent of precisely one of the variables $x_j^0$ or $x_j^1$.
If $n=3$, the structure of $\mathcal{F}^k$ for $k=3$, $2$, and $1$ is illustrated in Figures~\ref{fig:splx1}--\ref{fig:splx3} in the next section.
The set $\{0,\dots, k\} \times \{0,1\}^{n-k}$ is viewed as set of vertices of a polytope in $\mathbb{R}^n$.
To each hyper-face of the polytope we associate a variable and to each vertex a function $F_j$ of the adjacent $n$ variables.
In the cases $k=0$ and $k=1$, the polytope is an $n$-dimensional cube, while for $k=n$ the polytope is an $n$-dimensional simplex.

For $2\leq k\leq n$ and $\alpha, \alpha_k,\dots, \alpha_n \in (0,\infty)$ we define
\begin{align} \nonumber
\Lambda^k_{\alpha,\alpha_{k},\dots,\alpha_n} :=\, & \int_r^R \int_{\mathbb{R}^{n-k+1}}
\int_{\mathbb{R}^{2n-2k}}
\int_{\mathbb{R}^{k}} \Big| \int_{\mathbb{\mathbb{R}}}\mathcal{F}^k\, h_{t\alpha_k}(x_k-p_k) dx_{k} \Big|
\\ \nonumber
& g_{t \alpha }(x_0+\dots +x_{k-1}+p_{k}+\dots +p_{n}) dx_0\dots dx_{k-1}
\\ \label{def:lambdaK}
& \Big( \prod_{i=k+1}^{n} g_{t \alpha_i }(x_i^{0} - p_i)g_{t \alpha_i }(x_i^{1} - p_i)
 dx_i^{0} dx_i^{1} \Big) dp_{k}\dots dp_n \frac{dt}{t} .
\end{align}
For $1\leq k\leq n$ and $\alpha,\alpha_k,\dots, \alpha_n \in (0,\infty)$ we define
\begin{align} \nonumber
\widetilde{\Lambda}^k_{\alpha,\alpha_{k},\dots,\alpha_n} :=\, & \int_r^R \int_{\mathbb{R}^{n-k+1}}
\int_{\mathbb{R}^{2n-2k}}
\Big| \int_{\mathbb{R}^{k}}\int_{\mathbb{R}} \mathcal{F}^k\, h_{t\alpha_k}(x_k-p_k) dx_{k}
\\ \nonumber
& h_{t \alpha }(x_0+\dots +x_{k-1}+p_{k}+\dots +p_{n}) dx_0\dots dx_{k-1}\Big|
\\ \label{def:lambdaK1}
& \Big( \prod_{i=k+1}^{n} g_{t \alpha_i }(x_i^{0} - p_i)g_{t \alpha_i }(x_i^{1} - p_i)
 dx_i^{0} dx_i^{1} \Big) dp_{k}\dots dp_n \frac{dt}{t} .
\end{align}
The differences between \eqref{def:lambdaK} and \eqref{def:lambdaK1} are the occurrence of $g_{t \alpha }$
versus $h_{t \alpha}$ and the position of the absolute value signs. Also, we have no need to define \eqref{def:lambdaK} for $k=1$. Observe the trivial identity
\begin{align*}
h = 2^{1/2} \, h_{2^{-1/2} } \ast g_{2^{-1/2} }.
\end{align*}
Therefore the left hand-side of \eqref{toshow:smooth} is bounded by
$$ 2^{1/2} \Lambda_{ 2^{-1/2} ,\,2^{-1/2}}^{n}. $$
The estimate \eqref{toshow:smooth} is then
a consequence of the following lemma.

All constants in what follows will depend on $n$ and $k$ and we write
$A\lesssim B$ if there exists a finite constant $C$ depending on $n$ and $k$
such that $A\leq C B$.

\begin{lemma}\label{lemma:mainlemma}
For any $2\leq k \leq n$ and any $\alpha, \alpha_k,\dots,\alpha_n \in [2^{-(n-k+1)/2},\infty)$ we have the estimates
\begin{align}\label{inductest}
{\Lambda}^{k}_{\alpha,\alpha_{k},\dots, \alpha_n},\, \widetilde{\Lambda}^{k}_{\alpha,\alpha_{k},\dots, \alpha_n} 
\,\lesssim\, \big( \alpha \alpha_k \dots \alpha_n \big)^{2} \Big(\log \frac{R}{r} \Big)^{1-2^{-k+1}} .
\end{align}
For $k=1$ and any $\alpha,\alpha_1,\ldots,\alpha_n\in(0,\infty)$ we have the estimate
\begin{align*}
\widetilde{\Lambda}^{1}_{\alpha,\alpha_{1},\dots, \alpha_n} \lesssim 1.
\end{align*}
\end{lemma}

\begin{proof}[Proof of Lemma~\ref{lemma:mainlemma}] We induct on $1\leq k\leq n$ and let us begin by establishing the inductive step.
Take $2\leq k \leq n$ and $\alpha,\alpha_k,\dots, \alpha_n \in [2^{-(n-k+1)/2},\infty)$.
We first reduce the desired bound on
$\widetilde{\Lambda}^k_{\alpha,\alpha_k,\dots, \alpha_n}$ to that on
${\Lambda}^k_{\alpha,\alpha_k,\dots, \alpha_n}$.
We can dominate pointwise
\begin{align}\label{dom}
|h(x)| \lesssim \int_1^\infty g_\beta(x) \beta^{-4}{d\beta}
\end{align}
for each $x\in \mathbb{R}$.
Indeed, the right hand-side of \eqref{dom} is comparable to $x^{-4}$ for large $|x|$. By the triangle inequality and \eqref{dom} we can then bound
\begin{align*}
\widetilde{\Lambda}^k_{\alpha,\alpha_k,\dots, \alpha_n} \lesssim \int_1^\infty & \Lambda^k_{\alpha\beta,\alpha_k,\dots, \alpha_n } \,\beta^{-4} d\beta .
\end{align*}
Assuming the estimate \eqref{inductest} for ${\Lambda}^k_{\alpha,\alpha_k,\dots, \alpha_n}$, the right hand side of the last display is integrable in $\beta$.
Since $\alpha\in [2^{-(n-k+1)/2},\infty)$ is arbitrary, it suffices to prove upper bounds for ${\Lambda}^k_{\alpha,\alpha_k,\dots, \alpha_n}$.

Now we apply the Cauchy-Schwarz inequality in the variable $t$, which yields
\begin{align*}
\big( \Lambda_{\alpha,\alpha_{k},\dots,\alpha_n}^{k} \big)^2 \leq \Big( \log \frac{R}{r} \Big) & \int_r^R \bigg( \int_{\mathbb{R}^{n-k+1}}
\int_{\mathbb{R}^{2n-2k}} \int_{\mathbb{R}^{k}} \Big| \int_{\mathbb{R} } \mathcal{F}^k
\, h_{t\alpha_k}(x_k-p_k) dx_k \Big| \\
& g_{t \alpha}(x_0+\dots +x_{k-1}+p_{k}+\dots +p_{n}) dx_0\dots dx_{k-1} \\
& \Big( \prod_{i=k+1}^{n} g_{t \alpha_i }(x_i^{0} - p_i)g_{t \alpha_i }(x_i^{1} - p_i)
 dx_i^{0} dx_i^{1} \Big) dp_{k}\dots dp_n \bigg)^2 \frac{dt}{t} .
\end{align*}
We expand the definition of $\mathcal{F}^k$ and for each fixed $t$ we apply the Cauchy-Schwarz
inequality in all remaining integration variables but $x_k$. This way we obtain
\begin{align}\label{afterCS}
\big( \Lambda_{\alpha,\alpha_{k},\dots,\alpha_n}^{k} \big)^2 &\leq \Big( \log \frac{R}{r} \Big) \int_r^R \mathcal{M}_t\, \mathcal{N}_t \, \frac{dt}{t} ,
\end{align}
where
\begin{align*}
\mathcal{M}_t := \ & \int_{\mathbb{R}^{n-k+1}}
\int_{\mathbb{R}^{2n-2k}} \int_{\mathbb{R}^k} \bigg( \int_{\mathbb{R}} \prod_{i=0}^{k-1} \prod_{(r_{k+1},\dots,r_n)\in\{0,1\}^{n-k}}F_i(x_0,\dots,x_{i-1},x_{i+1},\dots, x_{k},x_{k+1}^{r_{k+1}},\dots ,x_{n}^{r_n})\\
& h_{t\alpha_k}(x_k-p_k) dx_k \bigg)^2 g_{t\alpha } (x_0+\dots +x_{k-1}+p_k+\dots + p_{n}) dx_0\dots dx_{k-1} \\
& \Big( \prod_{i=k+1}^{n} g_{t\alpha_i}(x_i^{0} - p_i)g_{t\alpha_i}(x_i^{1} - p_i)
 dx_i^{0} dx_i^{1} \Big) dp_{k}\dots dp_n
\end{align*}
and
\begin{align*}
\mathcal{N}_t
 := \ & \int_{\mathbb{R}^{n-k+1}}
\int_{\mathbb{R}^{2n-2k}}\int_{\mathbb{R}^k} \prod_{(r_{k+1},\dots,r_n)\in\{0,1\}^{n-k}} F_k(x_0,\dots,x_{k-1},x_{k+1}^{r_{k+1}},\dots,x_{n}^{r_{n}})^2 \\
& g_{t\alpha}(x_0+\dots +x_{k-1}+p_k+\dots + p_{n}) dx_0\dots dx_{k-1}\\
& \Big( \prod_{i=k+1}^{n} g_{t\alpha_i}(x_i^{0} - p_i)g_{t\alpha_i}(x_i^{1} - p_i)
 dx_i^{0} dx_i^{1} \Big) dp_{k}\dots dp_n .
\end{align*}

To estimate $\mathcal{N}_t$ pointwise for each fixed $t$, we first integrate in $p_k$ getting rid of $g_{t \alpha}$, then introduce the variables $y_i$ and $q_i$ via $x_i^0=x_i^1-y_i$ and $p_i = x_i^{1} - q_i$, respectively. 
Next, we apply H\"older's inequality in variables $x_0$ through $x_{k-1}$ and $x_{k+1}^{1}$ through $x_n^1$. 
Finally, we integrate the remaining Gaussian factors in $y_i$ and $q_i$ for $k+1 \leq i \leq n$. This yields 
\begin{align} \label{Nt}
\mathcal{N}_t\leq \|F_k^2\|_{{2^{n-k}}}^{2^{n-k}} =\|F_k\|_{{2^{n-k+1}}}^{2^{n-k+1}} =1,
\end{align}
so we have obtained an estimate which is uniform in $t>0$.

It remains to control
\begin{align}\label{integralmt}
\int_r^R
\mathcal{M}_t \frac{dt}{t} .
\end{align}
Expanding the square in the definition of $\mathcal{M}_t$, the expression \eqref{integralmt} becomes
the special case $k=j$ of the following more general expressions defined for $j\geq k$:
\begin{align}\label{defthetaj}
{\Theta}^{(j)}:= & \int_r^R \int_{\mathbb{R}^{n-k+1}}
\int_{\mathbb{R}^{2n-2k+2}} \int_{\mathbb{R}^k} \mathcal{F}^{k-1}
 \\
\nonumber
& g_{t\alpha}(x_0+\dots +x_{k-1}+p_k+\dots + p_{n}) dx_0\dots dx_{k-1}
\\ \nonumber
 & h_{t\alpha_j}(x_j^{0}-p_j) {h}_{t\alpha_j}(x_j^{1}-p_j)dx_j^{0} dx_j^{1} \Big( \prod_{\substack{i=k\\i\neq j}}^{n} g_{t\alpha_i}(x_i^{0} - p_i)g_{t\alpha_i}(x_i^{1} - p_i)
 dx_i^{0} dx_i^{1} \Big) dp_{k}\dots dp_n \frac{dt}{t}.
\end{align}
Also define
\begin{align*}
{\Theta} := \,
-\Big(1 + \alpha^{-2} \sum_{j=k}^{n} \alpha_j^2 \Big)
& \int_r^R \int_{\mathbb{R}^{n-k+2}} \int_{\mathbb{R}^{2n-2k+2}} \int_{\mathbb{R}^{k-1}} \int_{\mathbb{R}} \mathcal{F}^{k-1} \, h_{t\alpha 2^{-1/2} }(x_{k-1}-p_{k-1}) dx_{k-1}\\
& h_{t\alpha 2^{-1/2} }(x_0+\dots +x_{k-2}+p_{k-1}+\dots + p_{n})
 dx_0\dots dx_{k-2} \\
& \Big( \prod_{i=k}^{n} g_{t\alpha_i}(x_i^{0} - p_i)g_{t\alpha_i}(x_i^{1} - p_i)
 dx_i^{0} dx_i^{1} \Big) dp_{k-1}\dots dp_n \frac{dt}{t}.
\end{align*}
We claim that
\begin{align}\label{telescoping}
{\Theta} + \sum_{j=k}^n {\Theta}^{(j)} \lesssim 1.
\end{align}

Before proving the claim, we show how it can be used to control ${\Theta}^{(k)}$.
Note that ${\Theta}^{(k)}$ is non-negative because the real-valued terms in the expression assemble into an integral of a square
that came from previous application of the Cauchy-Schwarz inequality. The terms ${\Theta}^{(j)}$ are also non-negative for each $j\geq k$; the argument is the same after renaming the variables.
Therefore, comparing the definitions of $\Theta$ and
$\widetilde{\Lambda}^{k}_{\alpha,\alpha_k,\dots, \alpha_n }$,
\begin{align*}
{\Theta}^{(k)} \leq \sum_{j=k}^n {\Theta}^{(j)} \lesssim 1 + |\Theta|
\leq 1 + \Big(1+ \alpha^{-2}\sum_{j=k}^{n} \alpha_j^2 \Big)
\,\widetilde{\Lambda}^{k-1}_{\alpha 2^{-1/2}, \alpha 2^{-1/2}, \alpha_k,\dots,\alpha_n}.
\end{align*}
By the induction hypothesis (i.e.\@ the statement for $k-1$), we may estimate this display further by
$$ \lesssim \alpha^{-2} \Big( \alpha^2 + \sum_{j=k}^{n} \alpha_j^2 \Big) (\alpha^2 \alpha_k \dots \alpha_n)^{2} \Big( \log \frac{R}{r} \Big)^{1-2^{-k+2}}
\lesssim (\alpha \alpha_k \dots \alpha_n)^{4} \Big( \log \frac{R}{r} \Big)^{1-2^{-k+2}}, $$
where we have estimated the sum of the squared alphas by their product.
We combine this estimate with \eqref{afterCS} and \eqref{Nt}. Multiplying with $\log(R/r)$ and taking the square root shows \eqref{inductest} for the given $k$, completing the induction step up to the verification of the claim \eqref{telescoping}.

To see this claim, we employ the Fourier transform which we normalize as
$$\widehat{f}(\xi):=\int_{\mathbb{R}}f(x)e^{-2\pi \mathbbm{i} x\xi} dx.$$
For fixed $x_0,\dots,x_{k-1},x_k^{0},x_k^{1},\dots,x_n^{0},x_n^{1}$ the integral in $p_k,\dots, p_n$ in ${\Theta}^{(j)}$
is the integral of the function
\begin{align*}
H(q,q_k^{0},q_k^{1},\dots, q_n^{0},q_n^{1} ):=
& g_{t\alpha}(q + x_0+\dots +x_{k-1}) \\
 & h_{t\alpha_j}(q_j^{0}-x_j^{0}) {h}_{t\alpha_j}(q_j^{1}-x_j^{1})\Big( \prod_{\substack{i=k\\ i \neq j}}^{n} g_{t\alpha_i}(q_i^{0} - x_i^{0})g_{t\alpha_i}(q_i^{1} - x_i^{1})
 \Big)
\end{align*}
over the $(n-k+1)$-dimensional subspace
$$\{(p_k+\dots +p_n,\, p_k ,\, p_k,\, p_{k+1},\, p_{k+1},\dots, p_n,\,p_n) : p_k,\dots, p_n\in \mathbb{R}\} $$
of $\mathbb{R}^{2n-2k+3}$. The orthogonal complement of this subspace is
$$\{(\eta,\, \xi_k,\, -\xi_k-\eta , \, \xi_{k+1},\, -\xi_{k+1}-\eta ,\dots, \xi_n, \, -\xi_n-\eta) : \eta, \xi_k,\dots, \xi_n\in \mathbb{R}\}.$$ 
The previously mentioned integral is equal to the integral of the Fourier transform of $H$ over this orthogonal complement, which in turn becomes
\begin{align}\nonumber
& \int_{\mathbb{R}^{n-k+2}} \widehat{g_{t\alpha}}(\eta)\widehat{h_{t\alpha_j}}(\xi_{j}) {\widehat{h_{t\alpha_j}}}(-\xi_{j}-\eta)\prod_{\substack{i=k\\i\neq j}}^n \widehat{g_{t\alpha_i}}(\xi_{i})\widehat{g_{t\alpha_i}}(\xi_{i}+\eta)\\ \label{thetafourier}
& \;e^{2\pi \mathbbm{i} \left( (x_0+\dots +x_{k-1})\eta - \sum_{i=k}^n (x_i^{0}\xi_i + x_i^{1}(-\xi_i-\eta)) \right)}\, d\eta d\xi_k \dots d\xi_n .
\end{align}
Quite similarly, the integral in $p_{k-1},\dots,p_n$ in $\Theta$ can be expressed as
\begin{align}\nonumber
& \int_{\mathbb{R}^{n-k+2}} {\widehat{h_{t\alpha 2^{-1/2} }}}(\eta) {\widehat{h_{t\alpha 2^{-1/2} }}}(-\eta)\prod_{i=k}^n \widehat{g_{t\alpha_i}}(\xi_{i})\widehat{g_{t\alpha_i}}(\xi_{i}+\eta)\\ \label{thetafourier2}
& \;e^{2\pi \mathbbm{i} \left( (x_0+\dots +x_{k-1})\eta - \sum_{i=k}^n (x_i^{0}\xi_i + x_i^{1}(-\xi_i-\eta)) \right)}\, d\eta d\xi_k \dots d\xi_n .
\end{align}
Now we state the crucial ``telescoping'' or ``integration by parts'' identity
\begin{align}
\nonumber
& \Big( 1+ \alpha^{-2} \sum_{j=k}^{n} \alpha_j^2 \Big)
\int_r^R {\widehat{h_{t\alpha 2^{-1/2} }}}(\eta) {\widehat{h_{t\alpha 2^{-1/2} }}}(-\eta)\prod_{i=k}^n \widehat{g_{t\alpha_i}}(\xi_{i})\widehat{g_{t\alpha_i}}(\xi_{i}+\eta) \frac{dt}{t} \\\nonumber
& + \sum_{j=k}^n \int_r^R \widehat{g_{t\alpha} }(\eta) \widehat{h_{t\alpha_j}}(\xi_{j}) {\widehat{h_{t\alpha_j}}}(-\xi_j-\eta) \prod_{\substack{i=k\\ i\neq j}}^n \widehat{g_{t\alpha_i}}(\xi_{i})\widehat{g_{t\alpha_i}}(\xi_{i}+\eta) \frac{dt}{t}\\ \label{ftc}
& = \pi \big( G_r(\eta,\xi_k,\dots, \xi_n) - G_R(\eta,\xi_k,\dots, \xi_n) \big),
\end{align}
where for $t>0$ we have denoted
$$G_t(\eta,\xi_k,\dots, \xi_n):= \widehat{g_{t\alpha } }(\eta) \prod_{j=k}^n \widehat{g_{t\alpha_j}}(\xi_{j})\widehat{g_{t\alpha_j}}(\xi_{j}+\eta). $$
To see this identity, we use the fundamental theorem of calculus, together with $\widehat{g}(\xi)=e^{-\pi\xi^2}$, which yields that the right hand side of the identity \eqref{ftc} equals
\begin{align*}
& - \int_r^R \pi t\partial_t (G_t(\eta,\xi_k,\dots, \xi_n)) \frac{dt}{t}\\
& = \int_r^R 2\pi^2 t^2 \Big(\alpha ^2\eta^2 + \sum_{j=k}^n \alpha_j^2(\xi_j^2 + (\xi_j+\eta)^2) \Big) \,G_t(\eta,\xi_k,\dots, \xi_n) \frac{dt}{t}.
\end{align*}
Using $\widehat{h}(\xi) = 2\pi \mathbbm{i} \xi \widehat{g}(\xi)$ gives
$$ \widehat{h}(t\alpha 2^{-1/2}\eta) \widehat{h}(t\alpha 2^{-1/2}(-\eta)) 
= (2\pi \mathbbm{i} t\alpha)^2 2^{-1} \eta (-\eta) \widehat{g}(t\alpha 2^{-1/2} \eta) \widehat{g}(t\alpha 2^{-1/2} \eta)
= 2\pi^2 t^2 \alpha^2 \eta^2 \widehat{g}(t\alpha \eta) $$
and
$$ \widehat{h}(t\alpha_j \xi_j) \widehat{h}(t\alpha_j (-\xi_j-\eta))
= 4\pi^2 t^2 \alpha_j^2 \xi_j (\xi_j+\eta) \widehat{g}(t\alpha_j \xi_j) \widehat{g}(t\alpha_j (\xi_j+\eta)), $$
so the left hand side of \eqref{ftc} becomes
\begin{align*}
& \int_r^R \Big( 1+ \alpha^{-2} \sum_{j=k}^{n} \alpha_j^2 \Big) 2\pi^2 t^2 \alpha^2 \eta^2 \,G_t(\eta,\xi_k,\dots, \xi_n) \frac{dt}{t} \\
& + \int_r^R \Big( \sum_{j=k}^n 4\pi^2 t^2 \alpha_j^2 \xi_j (\xi_j+\eta) \Big) \,G_t(\eta,\xi_k,\dots, \xi_n) \frac{dt}{t}.
\end{align*}
A straightforward polynomial identity finally establishes \eqref{ftc}.

The terms on the left hand side of \eqref{ftc} correspond to the terms on the left hand side of \eqref{telescoping}: one only needs to multiply \eqref{ftc} with $\mathcal{F}^{k-1}$ and the complex exponential from \eqref{thetafourier}, \eqref{thetafourier2}, and perform the remaining integrations.
We thus need to show that the corresponding terms for the right hand side of \eqref{ftc} can be bounded by a constant. However, for $t=r$ or $t=R$ we have
\begin{align} \nonumber
\Big| &
\int_{\mathbb{R}^{n-k+1}}
\int_{\mathbb{R}^{2n-2k+2}} \int_{\mathbb{R}^k} \mathcal{F}^{k-1}\\ \nonumber
& g_{t\alpha}(x_0+\dots +x_{k-1}+p_k+\dots + p_{n}) dx_0\dots dx_{k-1}
\\ \label{single-scale2}
 & \Big( \prod_{\substack{i=k}}^{n} g_{t\alpha_i}(x_i^{0} - p_i)g_{t\alpha_i}(x_i^{1} - p_i)
 dx_i^{0} dx_i^{1} \Big) dp_{k}\dots dp_n \Big| \leq \|F_0\|_{{2^n}}^{2^{n-k+1}}\prod_{i=1}^{k-1} \|F_i\|^{2^{n-k+1}}_{{2^{n-i+1}}} = 1,
\end{align}
i.e.\@ these single-scale estimates are uniform in $t>0$ and $\alpha_i>1$. This follows by first introducing new variables $y$, $y_i$, and $q_i$ via $x_{0} = y-x_1-x_2-\dots - x_{k-1}$, $x_i^{0} = x_i^{1}-y_i$, and $p_i=x_i^{1}-q_i$. With these new variables, we first apply H\"older's inequality in $x_1,\dots ,x_{k-1}$, then integrate in $y$, then apply H\"older's inequality in $x_k^{1},\dots, x_n^{1}$, and finally integrate in $y_i$ and $q_i$ for $k\leq i \leq n$.

Inserting \eqref{ftc} into \eqref{thetafourier} and \eqref{thetafourier2}, passing to the spatial side and using the estimate \eqref{single-scale2} we obtain
the desired claim \eqref{telescoping}. This completes the proof of the inductive step.

It remains to establish the base case $k=1$ of the induction, i.e.\@ to estimate $\widetilde{\Lambda}^1_{\alpha,\alpha_k,\dots, \alpha_n}$. Unlike in the inductive step we do not dominate one of the functions $h$. Instead we apply the Cauchy-Schwarz inequality to \eqref{def:lambdaK1}
immediately in such a way that each of the terms on the right hand side invokes cancellative functions $h$. This is possible only in the case $k=1$ because here the integration in the variables $x_0$ and $x_1$ separates.
More precisely, we apply the Cauchy-Schwarz inequality in the integrals over the variables
$x_2^{0},x_2^{1},\dots x_n^{0,}x_n^{1}$, $p_1,\dots ,p_n$, and $t$ to obtain
\begin{align}\label{eq:csbase}
\big( \widetilde{\Lambda}^{1}_{\alpha,\alpha_1,\dots, \alpha_n} \big)^2 \leq \widetilde{\Theta}^{(1)}(F_0)\widetilde{\Theta}^{(1)}(F_1),
\end{align}
where for $1\leq j \leq n$ and a Schwartz function $F$ on $\mathbb{R}^n$ we have set
\begin{align*}
\widetilde{\Theta}^{(j)}(F):= &\int_r^R \nonumber \int_{\mathbb{R}^{n}}
\int_{\mathbb{R}^{2n}} \prod_{(r_{1},\dots,r_n)\in\{0,1\}^{n}}F (x_{1}^{r_{1}},\dots,x_{n}^{r_{n}}) \\
&\, h_{t \alpha_j} (x_j^{0}-p_{j}) {h}_{t\alpha_j} (x_j^{1}-p_{j}) dx_j^{0} dx_j^{1} \Big( \prod_{\substack{i=1\\i\neq j}}^{n} g_{t \alpha_i }(x_i^{0} - p_i)g_{t \alpha_i }(x_i^{1} - p_i)
 dx_i^{0} dx_i^{1} \Big) dp_{1}\dots dp_n \frac{dt}{t} .
\end{align*}

Similarly as in the inductive step, we now have
\begin{align}\label{telescoping2}
\sum_{j=1}^n \widetilde{\Theta}^{(j)}(F) \lesssim 1
\end{align}
for any $F$ with $\|F\|_{{2^n}} =1$.
Namely, $\widetilde{\Theta}^{(j)}(F)$ coincides with $\Theta^{(j)}$ for $k=1$, except for the choice of functions $F$ making up $\mathcal{F}^{k-1}$. 
Moreover,  $\mathcal{F}^{0}$ does not depend on $x_0$, so the integral in $x_0$ is merely the integral of a Gaussian. Likewise, the integral in $x_0$ in the definition of $\Theta$ for $k=1$ is an integral over the derivative of a Gaussian and hence vanishes. Thus claim \eqref{telescoping2} follows analogously to claim \eqref{telescoping}.

It remains to observe that $\widetilde{\Theta}^{(j)} \geq 0$ for each $1\leq j \leq n$, which is again analogous to the proof of the inductive step: simply observe that we are integrating squares of real-valued expressions.
Together with \eqref{telescoping2} this implies
$$\widetilde{\Theta}^{(1)}(F_0),\,\widetilde{\Theta}^{(1)}(F_1) \lesssim 1, $$
which by \eqref{eq:csbase} concludes the proof of the base case $k=1$ of the induction.
\end{proof}

\section{An illustration of the induction steps}
Figures~\ref{fig:splx1}--\ref{fig:splx3} represent the induction scheme for $n=3$.
The polyhedra in Figures~\ref{fig:splx1}--\ref{fig:splx3} represent the structure of $\mathcal{F}^k$
for $k=3$, $2$, and $1$ in this order. The vertices represent the various factors $F_j$ in the definition of $\mathcal{F}^k$, while the
faces represent the arguments in these factors, such that adjacency of a face to a vertex means that the argument appears
in the corresponding factor of $\mathcal{F}^k$.

The passage from left to right polyhedron in each figure represents the effect of the Cauchy-Schwarz inequality \eqref{afterCS}, passing
from a form $\Lambda^k_{\alpha,\alpha_k,\dots ,\alpha_n}$ involving $\mathcal{F}^{k}$ on the left to a form
$\mathcal{M}_t$ or $\Theta^{(k)}$ involving $\mathcal{F}^{k-1}$ on the right.

The shaded faces of the left polyhedra correspond to the variable $x_k$ in $\Lambda^k_{\alpha,\alpha_k,\dots ,\alpha_n}$
appearing in the cancellative function $h$. On the right hand side this variable has bifurcated into two variables $x_k^0$ and $x_k^1$ in
$\Theta^{(k)}$, both of which still carry cancellation.
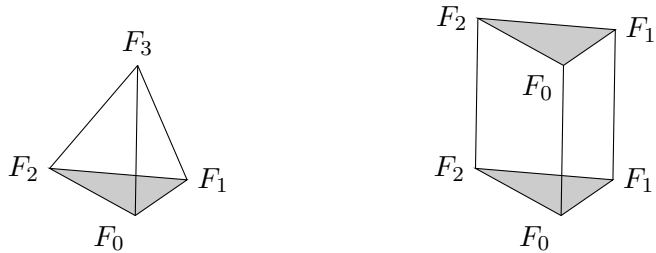
\begin{figure}[htb]
\centering
\begin{tikzpicture}[scale=1.3]
\begin{scope}[rotate = 45]
\coordinate [label=below left:$F_0$] (A) at (0,0,0);
\coordinate [label=right:$F_1$] (B) at (1,0.258819,0.965926);
\coordinate [label=left:$F_2$] (C) at (0,1.22474,0.707107);
\coordinate [label=above:$F_3$] (D) at (1,0.965926,-0.258819);
\draw (D)--(A);
\draw (D)--(B);
\draw (D)--(C);
\draw[fill=black, fill opacity=.2] (A)--(B)--(C)--cycle;
\end{scope}
\end{tikzpicture}
\qquad \qquad \qquad
\begin{tikzpicture}[scale=1.3]
\begin{scope}[rotate = 45]
\coordinate [label=below left:$F_0$] (A) at (0,0,0);
\coordinate [label=right:$F_1$] (B) at (1,0.258819,0.965926);
\coordinate [label=left:$F_2$] (C) at (0,1.22474,0.707107);
\coordinate [label=below left:$F_0$] (E) at (1,0.965926,-0.258819);
\coordinate [label=right:$F_1$] (F) at (2,1.22474,0.707107);
\coordinate [label=left:$F_2$] (G) at (1,2.19067,0.448289);
\draw (A)--(E);
\draw (B)--(F);
\draw (C)--(G);
\draw[fill=black, fill opacity=.2] (A)--(B)--(C)--cycle;
\draw[fill=black, fill opacity=.2] (E)--(F)--(G)--cycle;
\end{scope}
\end{tikzpicture}
\caption{Case $k=3$.} \label{fig:splx1}
\end{figure}

Comparing the right polyhedron in one figure to the left polyhedron in the next figure, the shaded faces move to a different location
indicating the effect of the telescoping estimate \eqref{telescoping}.
Note that the picture depicts only the most important of, in general many, terms in the telescoping
identity. In all but the last figure we have only one shaded face on the left polyhedron, since after domination of one function
$h$ by Gaussians only one function $h$ survives.
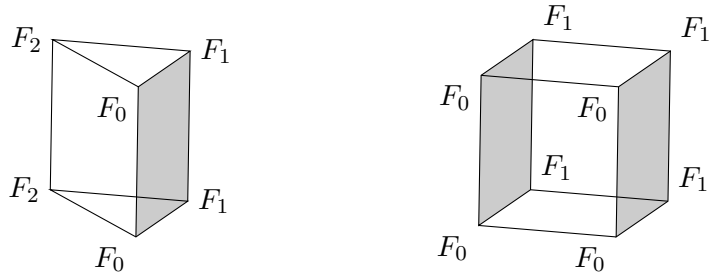
\begin{figure}[htb]
\begin{tikzpicture}[scale=1.3]
\begin{scope}[rotate = 45]
\coordinate [label=below left:$F_0$] (A) at (0,0,0);
\coordinate [label=right:$F_1$] (B) at (1,0.258819,0.965926);
\coordinate [label=left:$F_2$] (C) at (0,1.22474,0.707107);
\coordinate [label=below left:$F_0$] (E) at (1,0.965926,-0.258819);
\coordinate [label=right:$F_1$] (F) at (2,1.22474,0.707107);
\coordinate [label=left:$F_2$] (G) at (1,2.19067,0.448289);
\draw (A)--(E);
\draw (B)--(F);
\draw (C)--(G);
\draw (A)--(B)--(C)--cycle;
\draw (E)--(F)--(G)--cycle;
\fill[black, opacity=.2] (A)--(B)--(F)--(E)--cycle;
\end{scope}
\end{tikzpicture}
\qquad \qquad \qquad
\begin{tikzpicture}[scale=1.3]
\begin{scope}[rotate = 45]
\coordinate [label=below left:$F_0$] (A) at (0,0,0);
\coordinate [label=above right:$F_1$] (B) at (1,0.258819,0.965926);
\coordinate [label=below left:$F_0$] (E) at (1,0.965926,-0.258819);
\coordinate [label=above right:$F_1$] (F) at (2,1.22474,0.707107);
\coordinate [label=below left:$F_0$] (A1) at (-1,0.965926,-0.258819);
\coordinate [label=above right:$F_1$] (B1) at (0,1.22474,0.707107);
\coordinate [label=below left:$F_0$] (E1) at (0,1.93185,-0.517637);
\coordinate [label=above right:$F_1$] (F1) at (1,2.19067,0.448289);
\draw (A)--(A1);
\draw (B)--(B1);
\draw (E)--(E1);
\draw (F)--(F1);
\draw[fill=black, fill opacity=.2] (A)--(B)--(F)--(E)--cycle;
\draw[fill=black, fill opacity=.2] (A1)--(B1)--(F1)--(E1)--cycle;
\end{scope}
\end{tikzpicture}
\caption{Case $k=2$.}
\end{figure}

The last figure corresponds to the base case, which is treated differently. On the one hand we have two shaded
faces of the left polyhedron, and on the other hand the Cauchy-Schwarz inequality does not change the geometry of the polyhedron,
but merely the labeling of the corners. This stabilization of the process is ultimately the reason that the recursion stops.
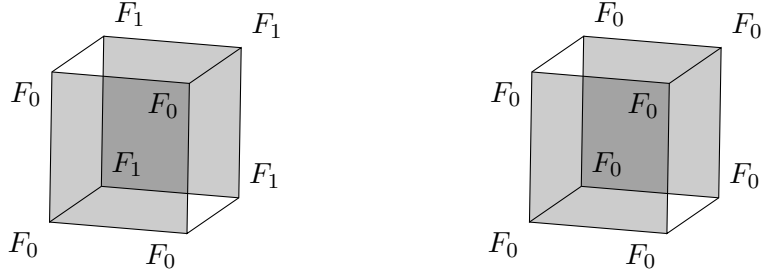
\begin{figure}[htb]
\begin{tikzpicture}[scale=1.3]
\begin{scope}[rotate = 45]
\coordinate [label=below left:$F_0$] (A) at (0,0,0);
\coordinate [label=above right:$F_1$] (B) at (1,0.258819,0.965926);
\coordinate [label=below left:$F_0$] (E) at (1,0.965926,-0.258819);
\coordinate [label=above right:$F_1$] (F) at (2,1.22474,0.707107);
\coordinate [label=below left:$F_0$] (A1) at (-1,0.965926,-0.258819);
\coordinate [label=above right:$F_1$] (B1) at (0,1.22474,0.707107);
\coordinate [label=below left:$F_0$] (E1) at (0,1.93185,-0.517637);
\coordinate [label=above right:$F_1$] (F1) at (1,2.19067,0.448289);
\draw (A)--(A1);
\draw (B)--(B1);
\draw (E)--(E1);
\draw (F)--(F1);
\draw (A)--(B)--(F)--(E)--cycle;
\draw (A1)--(B1)--(F1)--(E1)--cycle;
\fill[black, opacity=.2] (A)--(A1)--(E1)--(E)--cycle;
\fill[black, opacity=.2] (B)--(B1)--(F1)--(F)--cycle;
\end{scope}
\end{tikzpicture}
\qquad \qquad \qquad
\begin{tikzpicture}[scale=1.3]
\begin{scope}[rotate = 45]
\coordinate [label=below left:$F_0$] (A) at (0,0,0);
\coordinate [label=above right:$F_0$] (B) at (1,0.258819,0.965926);
\coordinate [label=below left:$F_0$] (E) at (1,0.965926,-0.258819);
\coordinate [label=above right:$F_0$] (F) at (2,1.22474,0.707107);
\coordinate [label=below left:$F_0$] (A1) at (-1,0.965926,-0.258819);
\coordinate [label=above right:$F_0$] (B1) at (0,1.22474,0.707107);
\coordinate [label=below left:$F_0$] (E1) at (0,1.93185,-0.517637);
\coordinate [label=above right:$F_0$] (F1) at (1,2.19067,0.448289);
\draw (A)--(A1);
\draw (B)--(B1);
\draw (E)--(E1);
\draw (F)--(F1);
\draw (A)--(B)--(F)--(E)--cycle;
\draw (A1)--(B1)--(F1)--(E1)--cycle;
\fill[black, opacity=.2] (A)--(A1)--(E1)--(E)--cycle;
\fill[black, opacity=.2] (B)--(B1)--(F1)--(F)--cycle;
\end{scope}
\end{tikzpicture}
\caption{Case $k=1$.}\label{fig:splx3}
\end{figure}


\section{Dyadic model of the simplex Hilbert transform}
\label{dyadicsection}
In this section we discuss the analogue of Theorem~\ref{thm:mainthm} for the dyadic model of the truncated simplex Hilbert transform.
Define
\begin{align*}
\Lambda_{n,m}^{\textup{d}} := \sum_{l=0}^{m-1} \sum_{(I_0,\dots,I_n)\in\mathcal{I}_l}
\epsilon_{l,I_0,\dots ,I_n}
\int_{(\mathbb{R}_+)^{n+1}} \prod_{i=0}^{n} F_i(x_0,\dots,x_{i-1},x_{i+1},\dots,x_n) 2^{-l} \Big( \prod_{i=0}^{n} \mathbbm{h}_{I_i}(x_i) dx_i \Big) ,
\end{align*}
where $n,m\geq 1$, $\mathbb{R}_+=[0,\infty)$, and for $l\in \mathbb{Z}$ we denote
$$ \mathcal{I}_l := \{(I_0,\dots,I_n) \,:\, 0\in I_0\oplus \dots \oplus I_n,\, I_i\;\textup{dyadic interval},\, I_i\subset\mathbb{R}_+,\; |I_i| = 2^l,\,1\leq i \leq n \}. $$
Here a dyadic interval is any interval of the form $[2^l m,2^l(m+1))$ with $m,l\in \mathbb{Z}$ and $\oplus$ is the addition of the Walsh group; see \cite{ktz:tht} for further details. The otherwise arbitrary coefficients $\epsilon_{l,I_0,\dots ,I_n}$ are assumed to be bounded in the absolute value by $1$ and
we have denoted by $\mathbbm{h}_{I}$ the $\textup{L}^\infty$-normalized Haar function on $I$.
A convenient property of the Haar functions is that
$$ \mathbbm{h}_{I_1\oplus I_2}(x_1\oplus x_2) = \mathbbm{h}_{I_1}(x_1) \mathbbm{h}_{I_2}(x_2) $$
whenever $I_1,I_2$ are dyadic intervals of the same length, $x_1\in I_1$, $x_2\in I_2$, and $I_1\oplus I_2$ is defined to be yet another dyadic interval of that same length whose left endpoint is the $\oplus$-sum of the left endpoints of $I_1$ and $I_2$.
Indeed, this is simply the character property of the more general Walsh functions.
In dyadic models it is common to replace $1/(x_0+\dots+x_n)$ with kernels such as
$$ K(x_0,\dots,x_n) = \sum_{l} \epsilon_{l} 2^{-l} \mathbbm{h}_{[0,2^l)}(x_1\oplus\dots\oplus x_n)
= \sum_{l} \sum_{(I_0,\dots,I_n)\in\mathcal{I}_l} \epsilon_{l} 2^{-l} \prod_{i=0}^{n} \mathbbm{h}_{I_i}(x_i). $$

This time the trivial estimate grows linearly in the number of scales $m$ and we want to improve on this trivial bound with a power less than one.

\begin{theorem}
There exists a finite constant $C$ depending only on $n\geq 1$ such that for any tuple $F_0,\dots, F_n$ of finite linear combinations of Haar functions and any $m\geq 1$ we have
\[ |\Lambda^{\textup{d}}_{n,m}| \leq C m^{1-2^{-n+1}} \|F_0\|_{{2^n}} \prod_{i=1}^n\|F_i\|_{{2^{n-i+1}}}. \]
\end{theorem}

\begin{proof}[Sketch of proof] Fix positive integers $n,m$ and functions $F_0,\dots, F_n$ normalized as in \eqref{norm}.
In order to perform the structural induction we introduce expressions indexed by $1\leq k\leq n$
\begin{align*}
\Lambda^{\textup{d},k} :=\, \sum_{l=0}^{m-1} \sum_{(I_0,\dots,I_n)\in\mathcal{I}_l}
& \int_{(\mathbb{R}_+)^{2n-2k}} \bigg| \int_{(\mathbb{R}_+)^{k+1}} \mathcal{F}^{k} (2^{-l})^{n-k+1} \Big( \prod_{i=0}^{k} \mathbbm{h}_{I_i}(x_i) dx_i \Big) \bigg|\\
& \Big( \prod_{i=k+1}^{n} \mathbbm{1}_{I_i}(x_i^{0}) \mathbbm{1}_{I_i}(x_i^{1}) dx_i^{0} dx_i^{1} \Big),
\end{align*}
where $\mathcal{F}^k$ is defined as in \eqref{fnk}.
We claim that
\begin{align}\label{est:dyad}
\Lambda^{\textup{d},k} \lesssim m^{1-2^{-k+1}}
\end{align}
for each $1\leq k \leq n$. Since $|\Lambda^{\textup{d}}_{n,m}|\leq \Lambda^{\textup{d},n}$, this then implies the theorem.

We prove \eqref{est:dyad} by induction on $k$ and begin with the inductive step.
Let $2\leq k \leq n$.
Performing the analogous steps from \eqref{afterCS} to \eqref{Nt} we obtain
\begin{align} \label{afterCSdyad}
\big(\Lambda^{\textup{d},k}\big)^2 \lesssim m \sum_{l=0}^{m-1} \mathcal{M}^{\textup{d}}_l,
\end{align}
where
\begin{align*}
\mathcal{M}^{\textup{d}}_l := \ & \sum_{(I_0,\dots,I_n)\in\mathcal{I}_l}
\int_{(\mathbb{R}_+)^{2n-k}} \bigg| \int_{\mathbb{R}_+} \prod_{i=0}^{k-1} \prod_{(r_{k+1},\dots,r_n)\in\{0,1\}^{n-k}} F_i(x_0,\dots,x_{i-1},x_{i+1},\dots,x_k,x_{k+1}^{r_{k+1}},\dots,x_{n}^{r_{n}})\\
& 2^{-l} \mathbbm{h}_{I_k}(x_k) dx_k \bigg|^2
 (2^{-l})^{n-k} \Big( \prod_{i=0}^{k-1} \mathbbm{1}_{I_i}(x_i) dx_i \Big) \Big( \prod_{i=k+1}^{n} \mathbbm{1}_{I_i}(x_i^{0}) \mathbbm{1}_{I_i}(x_i^{1}) dx_i^{0} dx_i^{1} \Big).
\end{align*}
Therefore it remains to control $\sum_{l=0}^{m-1} \mathcal{M}^{\textup{d}}_l$, which can be rewritten, in analogy with display \eqref{defthetaj}, as
\begin{align}\label{form:mjd}
& \sum_{l=0}^{m-1}\sum_{(I_0,\dots,I_n)\in\mathcal{I}_l}
\int_{(\mathbb{R}_+)^{2n-k+2}} \mathcal{F}^{k-1} \\ \nonumber
& (2^{-l})^{n-k+2} \Big( \prod_{i=0}^{k-1} \mathbbm{1}_{I_i}(x_i) dx_i \Big)
\Big( \mathbbm{h}_{I_k}(x_k^{(0)}) \mathbbm{h}_{I_k}(x_k^{(1)}) dx_k^{(0)} dx_k^{(1)} \Big)
\Big( \prod_{i=k+1}^{n} \mathbbm{1}_{I_i}(x_i^{(0)}) \mathbbm{1}_{I_i}(x_i^{(1)}) dx_i^{(0)} dx_i^{(1)} \Big) .
\end{align}
The identity \eqref{ftc} is now replaced by the dyadic ``telescoping'' identity
\begin{align}\nonumber
& \sum_{(I_0,\dots,I_n)\in\mathcal{I}_l} \bigg( \Big( \prod_{i=0}^{k-1} \mathbbm{h}_{I_i}(x_i) \Big)
\Big( \prod_{i=k}^{n} \big( \mathbbm{1}_{I_i}(x_i^{(0)}) \mathbbm{h}_{I_i}(x_i^{(1)}) + \mathbbm{h}_{I_i}(x_i^{(0)}) \mathbbm{1}_{I_i}(x_i^{(1)}) \big) \Big) \\ 
\nonumber
& \qquad\qquad\quad + \Big( \prod_{i=0}^{k-1} \mathbbm{1}_{I_i}(x_i) \Big)
\Big( \prod_{i=k}^{n} \big( \mathbbm{1}_{I_i}(x_i^{(0)}) \mathbbm{1}_{I_i}(x_i^{(1)}) + \mathbbm{h}_{I_i}(x_i^{(0)}) \mathbbm{h}_{I_i}(x_i^{(1)}) \big) \Big) \bigg) \\ \label{teldyad}
& = 2^{n-k+2} \sum_{(I_0,\dots,I_n)\in\mathcal{I}_{l-1}}
\Big( \prod_{i=0}^{k-1} \mathbbm{1}_{I_i}(x_i) \Big)
\Big( \prod_{i=k}^{n} \mathbbm{1}_{I_i}(x_i^{(0)}) \mathbbm{1}_{I_i}(x_i^{(1)}) \Big).
\end{align}
In order to verify it, we split each interval $I_i$ on the left hand side into its left ``child'' $I_i^{0}$ and its right ``child'' $I_i^{1}$, so that \eqref{teldyad} turns into
\begin{align*}
& \frac{1}{2}\sum_{(I_0,\dots,I_n)\in\mathcal{I}_l} \bigg( \Big( \prod_{i=0}^{k-1} \big(\mathbbm{1}_{I_i^0}(x_i)-\mathbbm{1}_{I_i^1}(x_i)\big) \Big)
\Big( \prod_{i=k}^{n} \big( \mathbbm{1}_{I_i^0}(x_i^{(0)}) \mathbbm{1}_{I_i^0}(x_i^{(1)}) - \mathbbm{1}_{I_i^1}(x_i^{(0)}) \mathbbm{1}_{I_i^1}(x_i^{(1)}) \big) \Big) \\
& \qquad\qquad\qquad + \Big( \prod_{i=0}^{k-1} \big(\mathbbm{1}_{I_i^0}(x_i)+\mathbbm{1}_{I_i^1}(x_i)\big) \Big)
\Big( \prod_{i=k}^{n} \big( \mathbbm{1}_{I_i^0}(x_i^{(0)}) \mathbbm{1}_{I_i^0}(x_i^{(1)}) + \mathbbm{1}_{I_i^1}(x_i^{(0)}) \mathbbm{1}_{I_i^1}(x_i^{(1)}) \big) \Big) \bigg) \\
& = \sum_{(I_0,\dots,I_n)\in\mathcal{I}_{l-1}}
\Big( \prod_{i=0}^{k-1} \mathbbm{1}_{I_i}(x_i) \Big)
\Big( \prod_{i=k}^{n} \mathbbm{1}_{I_i}(x_i^{(0)}) \mathbbm{1}_{I_i}(x_i^{(1)}) \Big).
\end{align*}
This identity becomes apparent once we observe that the tuple $(I_0^{s_0},\dots,I_n^{s_n})$ for some $(s_0,\dots,s_n)\in\{0,1\}^{n+1}$ belongs to $\mathcal{I}_{l-1}$ if and only if the number of $s_i$ that are equal to $1$ is even.

What we have in \eqref{form:mjd} can be recognized as one of the terms beginning with $\mathbbm{1}$'s in \eqref{teldyad}, after multiplying \eqref{teldyad} by $\mathcal{F}^{k-1}$, integrating and finally summing over the intervals and $l$. All terms in the second line of \eqref{teldyad} lead to non-negative expressions analogous to \eqref{defthetaj}, so it suffices to control their sum. What remains after summing the above identity in $l$, up to single-scale quantities analogous to \eqref{single-scale2}, are the terms beginning with $\mathbbm{h}$'s. By the triangle inequality, these terms lead to at most $2^n$ times
\begin{align*}
 \sum_{l=0}^{m-1} \sum_{(I_0,\dots,I_n)\in\mathcal{I}_l}
& \int_{(\mathbb{R}_+)^{2n-2k+2}} \bigg| \int_{(\mathbb{R}_+)^{k}} \mathcal{F}^{k-1} \, (2^{-l})^{n-k+2} \Big( \prod_{i=0}^{k-1} \mathbbm{h}_{I_i}(x_i) dx_i \Big) \bigg| \\
& \Big( \prod_{i=k}^{n} \mathbbm{1}_{I_i}(x_i^{(0)}) \mathbbm{1}_{I_i}(x_i^{(1)}) dx_i^{(0)} dx_i^{(1)} \Big) ,
\end{align*}
which can be recognized as $\Lambda^{\textup{d},k-1}$. Applying the induction hypothesis combined with \eqref{afterCSdyad} finishes the inductive step.

The base case $k=1$ can be deduced similarly as in the previous section.
\end{proof}


\section*{Acknowledgments}
P. D. and C. T. are supported by the Hausdorff Center for Mathematics.
V. K. is supported in part by the Croatian Science Foundation under the project 3526.


\end{document}